\documentclass{amsart}
\makeatletter
\newcommand*{\rom}[1]{\expandafter\@slowromancap\romannumeral #1@}
\makeatother
\usepackage{ifpdf}
\usepackage{amsmath, amsthm, amsfonts, amssymb, amscd}
\usepackage[active]{srcltx}
\usepackage{color}

    \ifpdf

      \usepackage[pdftex]{graphicx}  

    \else

      \usepackage[dvips]{graphicx}  

      \newcommand{\href}[2]{#2}

    \fi

\newtheorem{theorem}{Theorem}[section] 

\newtheorem{lemma}[theorem]{Lemma}

\newtheorem*{proposition*}{Proposition}

\newtheorem*{question*}{Question}
\newtheorem*{theorem*}{Theorem}
\newtheorem*{claim*}{Claim}
\newtheorem*{thmI}{Theorem \rom{1}}
\newtheorem*{thmII}{Theorem \rom{2}}
\newtheorem*{corollary*}{Corollary}

\theoremstyle{definition}

\theoremstyle{remark}
\newtheorem{remark}[theorem]{Remark}

\numberwithin{equation}{section}
\newcommand{\til}{\tilde} 

\newcommand{\homeo}{\operatorname{Homeo}}
\newcommand{\disp}{\phi}
\newcommand{\dist}{\mathrm{dist}}
\newcommand{\uwidth}{\operatorname{uw}}
\newcommand{\lwidth}{\operatorname{lw}}
\newcommand{\wind}{\operatorname{W}}
\newcommand{\diamup}{\operatorname{D}}

\newcommand{\R}{\mathbb{R}}\newcommand{\N}{\mathbb{N}}
\newcommand{\Z}{\mathbb{Z}}
\newcommand{\T}{\mathbb{T}}
\newcommand{\A}{\mathbb{A}}
\renewcommand{\SS}{\mathbb{S}}

\newcommand{\card}{\#} 
\newcommand{\sm}{\setminus}

\newcommand{\li}{\hat}
\newcommand{\mc}{\mathcal}

\newcommand{\abs}[1]{\big\lvert{#1}\big\rvert}

\newcommand{\ie}{i.e.\ }

\DeclareMathOperator{\pr}{\rm{pr}}
\DeclareMathOperator{\bd}{\partial}
\DeclareMathOperator{\diam}{\rm{diam}}

\begin{document}

\author[A. Koropecki]{Andres Koropecki}
\address{Instituto de Matem\'atica e Estat\'\i stica, Universidade Federal Fluminense}
\curraddr{Rua Prof. Marcos Waldemar de Freitas Reis, S/N - Bloco H, 4o Andar, Campus do Gragoat\'a, Niterói, RJ, Brasil}
\email{ak@id.uff.br}

\author[A. Passeggi]{Alejandro Passeggi}
\address{UdelaR, Facultad de Ciencias.}
\curraddr{Igua 4225 esq. Mataojo. Montevideo, Uruguay.}
\email{alepasseggi@gmail.com}

\author[M. Sambarino]{Mart\'{\i}n Sambarino}
\address{UdelaR, Facultad de Ciencias.}
\curraddr{Igua 4225 esq. Mataojo. Montevideo, Uruguay.}
\email{samba@cmat.edu.uy}

\title[The Franks-Misiurewicz conjecture]{The Franks-Misiurewicz conjecture\\ for extensions of irrational rotations}

\begin{abstract}
We show that a toral homeomorphism which is homotopic to the identity and topologically semiconjugate to an irrational rotation of the circle is always a pseudo-rotation (\ie its rotation set is a single point). In combination with recent results, this allows us to complete the study of the Franks-Misiurewicz conjecture in the minimal case.
\end{abstract}

\maketitle

\section{Introduction}

It is a general goal in mathematics to classify objects by means of simpler invariants associated to them. In the study of the dynamics of surface maps, the rotation set is a prototypical example of this approach. Being a natural generalization in different contexts of the Poincar\'{e} rotation number of orientation preserving circle homeomorphisms, it provides basic dynamical information for surface maps in the homotopy class of the identity \cite{m-z, pollicott, franks-rot-gen}.

In the two dimensional torus, it can be said that a theory has emerged supported on this invariant (see \cite{m-z} for a wide exposition). If $F$ is a lift of a torus homeomorphism $f$ in the homotopy class of the identity, its rotation set is defined by
\small
\begin{equation*}
\rho(F)=\left\{\lim_{i\to \infty}\frac{F^{n_i}(x_i)-x_i}{n_i}:\mbox{ where }n_i\nearrow+\infty,\ x_i\in\R^2\right\}.
\end{equation*}
\normalsize

In the seminal article \cite{m-z} Misiurewicz and Ziemian proved the convexity and compactness of rotation sets.
The finite nature for the possible geometries of a convex set in the plane given by points, non-trivial line segments, or convex sets with nonempty interior, allowed to start a systematic study based on these three cases.

Results concerning this theory can be classified in two different directions. A first direction aims to obtain interesting dynamical information from knowing the geometry of the rotation set, where the list of results is huge. For instance it is known that when the rotation set has nonempty interior the map has positive topological entropy \cite{llibre-mackay} and abundance of periodic orbits and ergodic measures \cite{franks-reali, franks-reali2, m-z2}; bounded deviations properties are found both for the non-empty interior case and the non-trivial segment case \cite{davalos, addas-zanata, Forcing,Kocsard} (see also \cite{beguinsurveytoro,PassThesis} as possible surveys\footnote{Unfortunately both surveys are far from being up to date}).

A second direction aims to is to establish which kind of convex sets can be realized as rotation sets.
Here we find fundamental problems which remain unanswered (compared with the first direction, it can be said, the state of art is considerably underdeveloped).
For convex sets having non-empty interior, all known examples achieved as rotation sets have countably many extremal points \cite{Kwapisznonpol,BoyCaHall}.
For rotation sets with empty interior, there is a long-standing conjecture due to Franks and Misiurewicz \cite{franksmisiu}, which is the matter of this work. The conjecture aims to classify the possible rotation sets with nonempty interior, and it states that any such rotation set is either a singleton or a non-trivial line segment $I$ which falls in one of the following cases:
\begin{itemize}
\item[(i)] $I$ has rational slope and contains rational points\footnote{\ie points with both coordinates rational};
\item[(ii)] $I$ has irrational slope and one of the endpoints is rational.
\end{itemize}

For case (ii) A. Avila has presented a counterexample in 2014\footnote{Still unpublished.}, where a non-trivial segment with irrational slope containing no rational points is obtained as rotation set. Moreover, the counterexample is minimal ($\T^2$ is the unique compact invariant set) and $C^\infty$, among other interesting features.

Still concerning case (ii), P. Le Calvez and F. A. Tal showed that whenever the rotation set is a non-trivial segment with irrational slope and containing a rational point, this rational point must be an endpoint of the segment \cite{Forcing}, so segments of irrational slope containing rational points obey the conjecture.

Item (i), however, remains open: is it true that the only nontrivial segments of rational slope realized as a rotation set are those containing rational points? Although partial progress has been made in recent years \cite{jager-linear, jagertal, JP, Kocsard}, the question remained open even in the minimal case.

In this article we prove that, in contrast to Avila's counter example, case (i) in the conjecture is true for minimal homeomorphisms. As we see in the next paragraph, we prove that case (i) must hold in the
family of  \emph{extensions of irrational rotations} which in particular provides the answer for minimal homeomorphisms.

\subsection{Precise statement, context and scope.}

The family of \emph{extensions of irrational rotations} is given by those toral homeomorphisms in the homotopy class of the identity which are topologically semi-conjugate to an irrational rotation of the circle. The study of the conjecture in this particular family was introduced in \cite{JP}, following a program by T. J\"ager: supported in the ideas presented in \cite{jager-linear}, one may first aim to show that every possible counter example for the rational case (i) in the conjecture must be contained in this family, and as a second step one may study the conjecture in the class of extensions of irrational rotations. There is significant progress in the first step of the  program under some recurrence assumptions \cite{jager-linear, jagertal, Kocsard}. On the other hand, for the second step the only known result states that if a counter-example exists, the fibers of the conjugation must be topologically complicated \cite{JP}. This sole fact does not lead to a contradiction, since such a fiber structure is possible for extensions of irrational rotations (see \cite{beguin-crovisier-jager}). Our main result in this article completely solves the second step of J\"ager's program: there are no counter examples to the Franks-Misiurewicz conjecture in the family of extensions of irrational rotations.

The rotation set of an extension of an irrational rotation in $\T^2$ contains no rational points, and it must be either a singleton or an interval of rational slope (see for instance \cite{JP}). In \cite{jagertal} it is proved that
every area-preserving homeomorphism homotopic to the identity having a \emph{bounded deviations} property is an extension of an irrational rotation (see also \cite{jager-linear}). Recently A. Kocsard showed that minimal homeomorphisms having a non-trivial interval with rational slope as rotation set have the \emph{bounded deviations} property \cite{Kocsard}, and as a consequence every minimal homeomorphism having a non-trivial interval with rational slope as rotation set must be an extension of an irrational rotation.


Our main result is the following:

\begin{thmI}\label{t.rotsetext}
The rotation set of a lift of any extension of an irrational rotation is a singleton.
\end{thmI}

Using the previously mentioned results we find that case (i) in the Franks-Misiurewicz conjecture is true for minimal homeomorphisms:

\begin{thmII}\label{c.rotsetmin}
The rotation set of a lift of any minimal homeomorphism of $\T^2$ homotopic to the identity is either:
\begin{itemize}
\item[(i)] a single point of irrational coordinates, or
\item[(ii)] a segment with irrational slope containing no rational points.
\end{itemize}
\end{thmII}

Note that both cases are realized; the first one by minimal rotations, and the second by Avila's example.

In the case of diffeomorphisms, J. Kwapisz has shown that the possible existence of an example whose rotation set is an interval contained in a line of irrational slope having a rational point outside the interval is equivalent to the existence of an example with a non-trivial segment of rational slope containing no rational points \cite{kwapisz-apriori}. Therefore, we have the following corollary.

\begin{corollary*} For diffeomorphisms, case (ii) of the previous theorem can only hold if the supporting line of the segment contains no rational points.
\end{corollary*}

Finally, it should be mentioned that one of the main theorems in \cite{kwapisz-apriori} contains the previous corollary, and moreover states that rotation sets which are intervals of rational slope having no rational points cannot be realized by diffeomorphisms. Unfortunately, there is a critical flaw in the proof\footnote{as acknowledged by the author}, which uses some convoluted estimations relying in quasiconformality properties and extremal length.

\subsection{Comments on the proof of Theorem I}

In \cite{JP} it is proved that an extension $f$ of an irrational rotation having an interval as rotation set has a semiconjugacy to an irrational rotation so that every fiber is an essential annular continuum, and almost every fiber contains points realizing both extremal rotation vectors.

In order to prove Theorem I we develop some techniques concerning the geometry of essential loops which have the property of remaining under iteration close enough to two points having different rotation vectors (see Section \ref{sec:dragging}). In Section \ref{sec:freq} we show that one can choose a topological annulus $A$ which contains at least two fibers of the semiconjugacy and whose ``width'' remains small enough after most iterations by $f$. Applying the results from Section \ref{sec:dragging}, we are able to show that every essential loop in $f^n(A)$ will contain arcs whose \emph{winding number} becomes arbitrarily large with $n$. This in turn will imply that $A$ is increasingly \emph{distorted}, and as a consequence of this distortion we show that the two boundary circles of $f^n(A)$ contain points arbitrarily close to each other.
This leads to a contradiction since the (pointwise) distance between two different fibers remains bounded below by a constant under iterations, due to the semiconjugation to a rigid rotation.

\subsection*{Acknowledgments:} We would like to thank Sylvain Crovisier for the fruitful conversations related to this article.

\section{preliminaries} \label{sec:prelim}
We denote by $\pr_1\colon \R^2\to \R$ and $\pr_2\colon \R^2\to \R$ the projections onto the first and second coordinates, respectively, and we define $T_1(x,y) = (x+1,y)$, $T_2(x,y) = (x,y+1)$. We consider the open annulus $\A$ defined as $R/\langle T_1\rangle$, and we let $\tau \colon \R^2\to \A$ be the covering projection. The vertical translation $T_2$ induces a vertical translation on $\A$ which we still denote $T_2$, and we consider the torus $\T^2=\A/\langle T_2\rangle$ with covering projection $\theta\colon \A\to \T^2$. Note that the map $\pi = \tau\circ \theta$ is a universal covering of $\T^2$. All spaces are endowed by the metric induced by the euclidean metric in $\R^2$.

For a surface $S$, we denote by $\homeo_0(S)$ the space of homeomorphisms of $S$ isotopic to the identity. Any $f\in \homeo_0(\T^2)$ can be lifted (by the covering $\theta$) to a homeomorphism $\hat f\in \homeo_0(\A)$ which commutes with $T_2$, and $\hat f$ in turn lifts to a homeomorphism $F\colon \R^2\to \R^2$ with commutes with both $T_1$ and $T_2$, and which is also a lift of $f$ (by the covering $\pi$).

For convenience, let us denote by $\disp_F\colon \A\to \R$ the \emph{horizontal displacement function} associated to $F$, defined on $\A$ by $\disp_F(x) = \pr_1(F(\til x) - \til x)$ for any $\til x\in \theta^{-1}(x)$. Note that since $z\mapsto F(z)-z$ is $\Z^2$-periodic, this definition is independent on the choice of $\til x$, and moreover $\disp_F$ is $T_2$-periodic. In particular it is bounded.
It is useful to note that $\disp_{F^n}(x) = \sum_{k=0}^{n-1} \disp_F(\li f^k(x)) = \pr_1(F^n(\til x) - \til x)$ for any $\til x\in \theta^{-1}(x)$.

\subsection{Some topological definitions and facts}

An arc in $\A$ from $x$ to $y$ is a continuous map $\sigma \colon [a,b] \to \A$ such that $\sigma(a) = x$ and $\sigma(b) = y$. Two arcs are equivalent if one is a reparametrization of the other (preserving the endpoints). We identify equivalent arcs. The arc is simple if the map $\sigma$ may be chosen injective. In the case of a simple arc, we often use the same notation for $\sigma$ and the image of $\sigma$. A loop is an arc $\gamma$ whose two endpoints coincide. In that case we say that $\gamma$ is simple if there is a parametrization $\gamma\colon[a,b]\to \A$ which is injective on $[a,b)$. The loop $\gamma$ is essential its complement in $\A$ has two unbounded components.

An \emph{essential continuum} $E\subset \A$ is a continuum such that $\A\sm E$ has two unbounded connected components, which we denote $\mc U^+(E)$ and $\mc U^-(E)$ (where $\mc U^+$ is the one unbounded above and $\mc U^-$ is the one unbounded below). We say that $E$ is an \emph{essential annular continuum} if $\A \sm E$ has exactly two connected components, both of which are unbounded.

A continuum $C\subset \T^2$ is called a \emph{horizontal (annular) continuum} if each connected component of $\theta^{-1}(C)$ is an essential (annular) continuum.
Similarly, an open or closed (topological) annulus $A\subset \T^2$ is called horizontal if each connected component of $\theta^{-1}(A)$ contains an essential loop.

If $X,Y$ are two sets in $\A$ or $\T^2$, we write $d(X,Y) = \inf\{d(x,y) : x\in X, y\in Y\}$. When $X$ is a singleton we write $d(x,Y)$ instead of $d(\{x\}, Y)$. By $B_r(X)$ we denote the $r$-neighborhood of $X$, \ie the set $\{y : d(y,X)<r\}$. If $X,Y$ are compact we denote by $d_H(X,Y)$ the Hausdorff distance between the two sets, \ie the infimum of all numbers $\epsilon>0$ such that $X\subset B_\epsilon(Y)$ and $Y\subset B_\epsilon(X)$.
The Hausdorff distance is a complete metric.

Given two essential continua $C_1, C_2$ in $\A$, we write $C_1\prec C_2$ if $C_1\subset \mc U^-(C_2)$. This defines a partial order.
The following lemma is contained in \cite[Lemma 3.8]{JKP}.
\begin{lemma}\label{lem:JKP} If a sequence of essential continua $(C_k)_{k\in \N}$ is increasing in the partial order $\prec$, then there is an essential continuum $C$ such that $d_H(C_k, C) \to 0$ as $k\to \infty$. Moreover, $C=\bd \bigcup_{k\in \N} \mc U^-(C_k)$. A similar property holds for a decreasing sequence.
\end{lemma}

Given an essential annular continuum $A\subset \A$, we define its \emph{upper width} as
$$\uwidth(A) = \sup\{d_H(C_1, C_2) : C_1, C_2 \text { are essential continua in $A$}\}$$
and its \emph{lower width} as
$$\lwidth(A) = \sup \{d(C_1, C_2) : C_1, C_2 \text { are essential continua in $A$}\}.$$
If $A$ is a closed topological annulus, one can easily verify that $\lwidth(A) = d(\bd^+ A, \bd^- A)$, where $\bd^+A$ and $\bd^-A$ are the two boundary components of $A$, and $\uwidth(A) = d_H(\bd^+ A, \bd^- A)$.
Note that we used the infimun distance in the first case and the Hausdorff distance for the second case.

We remark that an equivalent definition of $\uwidth(A)$ is as the smallest number $r>0$ such that for every essential continuum $C\subset A$ one has $A\subset B_\epsilon(C)$.
Note also that if $A\subset A'$ then $\uwidth(A)\leq \uwidth(A')$ and $\lwidth(A)\leq \lwidth(A')$.

If $A\subset \T^2$ is a horizontal annular continuum, we define its upper and lower width as the upper and lower width of any lift of $A$ to $\A$, respectively (and this is independent of the choice of the lift).

\section{Extensions of irrational rotations}
\label{sec:freq}

Let us say that $h\colon \T^2\to \SS^1$ is a \emph{horizontal map} if $h$ is continuous, surjective, and $h^{-1}(t)$ is a horizontal annular continuum for each $t\in \SS^1$.
Given $f\in \homeo_0(\T^2)$, we say that $f$ is a \emph{horizontal extension of an irrational rotation} if there exists a horizontal map $h$ such that $hf = Rh$, where $R$ is an irrational rotation of $\SS^1$.

We will use the following result due to T. J\"ager and the second author of this article \cite{JP}:
\begin{theorem} If $f$ is an extension of an irrational rotation, then $f$ is topologically conjugate to a horizontal extension of an irrational rotation.
\end{theorem}

As mentioned in the introduction, the proof of Theorem I is based in showing that if we iterate certain annular neighborhood $A$ of a fiber of the horizontal semiconjugacy, then the boundary components get arbitrarily close. A proof of this fact would be easier provided we know that every fiber of the semiconjugacy has small width (bounded above by the continuity module of $\frac{1}{4}$ for $f$), which would be true for instance if every fiber was a \emph{circloid}\footnote{A minimal annular continuum with respect to the inclusion}. Unfortunately we can not ensure this fact, instead we will strongly use that only finitely many of the fibers can have large width, which in turn will imply that under iterations the boundary components of $A$ are in a small Hausdorff distance for a high frequency of iterations.

For the remainder of this section, fix a horizontal extension of an irrational rotation $f\in \homeo_0(\T^2)$, and let $h$ be the horizontal map such that $hf = Rh$ where $R$ is an irrational rotation. We also fix a lift $\li f\in \homeo_0(\A)$ of $f$ and a lift $F\colon \R^2\to \R^2$ of $\li f$.

Our main purpose in this section is to show the following result, which enumerates the key properties which will be used in the proof of our main theorem. Recall that the lower density of a set $G\subset \N$ is defined as $\liminf_{n\to \infty} \card\{k\in G : k\leq n\}/n$. We state the lemma in the annulus $\A$ since we will work in that setting later.

\begin{lemma}\label{lem:todo} Suppose that $\rho(F) = [\rho^-, \rho^+]\times \{\alpha\}$. Then, given $\delta>0$ and $\epsilon>0$, there exists a closed essential topological annulus $A\subset \A$, an essential simple loop $\gamma\subset \A \sm A$, two points $x,y\in \A\sm A$, a set $G\subset \Z$ and $B = B(f)>0$ such that
\begin{itemize}
\item[(1)] $\disp_{F^n}(x)/n \to \rho^-$ and $\disp_{F^n}(y)/n \to \rho^+$ as $n\to \infty$;
\item[(2)] If $n\in G$, then $d(\li f^n(x), \li f^n(\gamma)) < \epsilon$ and $d(\li f^n(y), \li f^n(\gamma))<\epsilon$;
\item[(3)] The lower density of $G$ is at least $1-\delta$;
\item[(4)] $\diam(\pr_2(\li f^n(A)))\leq B$ for all $n\in\N$;
\item[(5)] $A$ separates $\{x,y\}$ from $\gamma$ in $\A$, and $$\sup_{n\in \N} \lwidth(\li f^n(A)) > 0.$$
\end{itemize}
\end{lemma}

Before proceeding to its proof, we need some results about the fibers of the map $h$.
Note that the family $\mc F$ of all fibers of $h$ is a decomposition of $\T^2$ into horizontal annular continua. From the continuity of $h$ follows that $\mc F$ is a \emph{upper semicontinuous decomposition}: if $C_n\in \mc F$ is a sequence of fibers such that $C_n\to C$ in the Hausdorff topology, then $C\subset C'$ for some $C'\in \mc F$. We also note that $h$ lifts to a map $H\colon \A \to \R$ whose fibers are the lifts of fibers of $h$, and choosing the orientation of $\R$ adequately we have that $H^{-1}(x) \prec H^{-1}(y)$ if and only if $x< y$.
Finally we remark that due to the fiber structure of $h$, whenever $I\subset \SS^1$ is an open interval, its preimage $A = h^{-1}(I)$ is an open topological annulus, and when $I\subset \SS^1$ is a closed interval $h^{-1}(I)$ is a horizontal annular continuum.

\begin{lemma}\label{lem:tec1} For each $\epsilon>0$ and $t\in \SS^1$ there exists a neighborhood $I_t$ of $t$ such that whenever $I\subset I_t\sm \{t\}$ is a closed interval one has $\uwidth(h^{-1}(I))< \epsilon$.
\end{lemma}
\begin{proof}
It suffices to prove the analogous claim on $\A$, \ie for each $t\in \R$ there exists a neighborhood $I_t$ of $t$ such that whenever $I\subset I_t\sm \{t\}$ is a closed interval one has $\uwidth(H^{-1}(I))< \epsilon$. Suppose this is not the case. Then there exists a sequence $J_n$ of closed intervals disjoint from $t$, converging to $t$, such that $\uwidth(H^{-1}(J_n)) \geq \epsilon$. For each $J_n$ we may find two essential continua $C_n^1, C_n^2\subset J_n$ such that $d_H(C_n^1, C_n^2) \geq \epsilon$. Passing to a subsequence we may assume that the intervals $J_n$ are either increasing or decreasing. We assume the former case, as the other case is analogous. This implies that both sequences $(C_n^i)_{n\in \N}$ are increasing in the order $\prec$. Thus by Lemma \ref{lem:JKP} we have $d_H(C_n^i, C^i) \to 0$ where  $C^i := \bd U^-_i$ and $U^-_i = \bigcup_{k\in \N} \mc U^-(C_k^i)$. But one easily verifies that $U^-_i = H^{-1}((-\infty, t])$, so $C^1 = C^2$. This implies that $d_H(C_n^1, C_n^2) \to 0$ as $n\to\infty$, a contradiction.
\end{proof}

\begin{remark} We note that as a consequence of the previous lemma, the set of all fibers of $h$ of positive upper width is greater than a given $\epsilon>0$ must be finite. Indeed the lemma implies that the set $\{t\in\SS^1 : \uwidth(h^{-1}(t)) \geq \epsilon\}$ has no accumulation points.
\end{remark}

\begin{lemma}\label{lem:density} Given $\epsilon>0$ and $\delta>0$, there exists $\eta>0$ such that for any closed interval $I\subset \SS^2$ of length smaller than $\eta$, if $A=h^{-1}(I)$ the set $\{n \in \N : \uwidth(f^n(A)) < \epsilon \}$ has lower density at least $1-\delta$.
\end{lemma}
\begin{proof} Consider a cover of $\SS^1$ by finitely many neighborhoods $I_{t_1}, \dots, I_{t_k}$ as in Lemma \ref{lem:tec1}, and let $0<\eta<\delta/k$ be such that whenever a closed interval $I$ has length smaller than $\eta$ one  has $I\subset I_{t_i}$ for some $i$. This means that any such $I$ satisfies $\uwidth(h^{-1}(I)) < \epsilon$ unless it contains $t_i$ for some $i$. If $G_i(I)\subset \N$ denotes the set of of all $n\in \N$ such that $t_i\notin R^n(I)$, we have from the ergodicity of the irrational rotation $R$ that $G_i(I)$ has lower density $1-\ell(I) > 1-\eta$, where $\ell(I)$ denotes the length of $I$. Hence the set $G(I) = \bigcap_{i=1}^k G_i(I)$ has lower density at least $1-k\eta > 1-\delta$. Note that since $R^n(I)$ has the same length as $I$, we have $\uwidth(h^{-1}(R^n(I))) <\epsilon$ whenever $n\in G(I)$. The proof is concluded noting that if $A=h^{-1}(I)$ then $f^n(A) = h^{-1}(R^n(I))$.
\end{proof}

\begin{lemma}\label{lem:lw-min} If $A=h^{-1}(I)$ for some nontrivial closed interval $I\subset \SS^1$, then $$\inf \{\lwidth(f^n(A)) : n\in \Z\} > 0.$$
\end{lemma}
\begin{proof}
Let $a,b$ be the endpoints of $I$, and $\epsilon = d(a,b)$.
Choose $\delta>0$ such that whenever $d(x,y)<\delta$ for $x,y\in \T^2$ one has $d(h(x), h(y))<\epsilon$. Note that this means that $d(h^{-1}(a), h^{-1}(b)) \geq \delta$. Moreover, since $d(R^n(a), R^n(b)) = d(a,b) = \epsilon$, we also have for any $n\in \N$
$$d(f^n(h^{-1}(a)), f^n(h^{-1}(b))) = d(h^{-1}(R^n(a)), h^{-1}(R^n(b))) \geq \delta.$$
Thus $C_a = f^n(h^{-1}(a))$ and $C_b = f^n(h^{-1}(b))$ are two horizontal continua in $f^n(A)$ such that $d(C_a, C_b) \geq \delta$, and it follows easily that $\lwidth(f^n(A)) \geq \delta$ for all $n\in \N$.
\end{proof}

\begin{lemma}\label{lem:altura} There exists $B>0$ such that for every interval $I\subset \SS^1$, if $A\subset \A$ is a lift of $A_0 = h^{-1}(I)$ then $\diam(\pr_2(\li{f}^n(A))) \leq B$ for all $n\in \Z$.
\end{lemma}
\begin{proof}
Fix $t\in \SS^1$, let $C$ be a lift to $\A$ of $h^{-1}(t)$. If $\mc A\subset \A$ is the annulus bounded by $C$ and its vertical translation by two, \ie $T^2_2(C)$, then for each $n$ there is $i\in \Z$ such that $\li f^n(A)\subset T^i(\mc A)$. Since $T$ is an isometry, $B = \diam(\pr_2(\mc A))$ satisfies the required property.
\end{proof}

\subsection{Proof of Lemma \ref{lem:todo}}

Let $\eta<1$ be as in Lemma \ref{lem:density}, let $\mc A_0 = h^{-1}(I_0)$ where $I_0\subset \SS^1$ is some closed interval of length smaller than $\eta$, and choose any lift $A_0\subset \A$ of $\mc A_0$. Note that $A_0$ is fibered by the fibers of $H$, \ie $A_0 = H^{-1}(I_0')$ for some interval $I_0'\subset \R$ (which is a lift of $I_0$). Noting also that the open region bounded by any pair of different fibers of $H$ is an open topological annulus, and in particular contains an essential simple loop, by an easy argument one obtains a loop $\gamma$ and two disjoint closed topological annuli $A, A' \subset A_0$ such that:
\begin{itemize}
\item $\gamma \prec A \prec A'$
\item $H^{-1}(I)\subset A$ for some nontrivial closed interval $I\subset I_0$;
\item $H^{-1}(J)\subset A'$ for some nontrivial closed interval $J\subset I_0$.
\end{itemize}
Since extremal points of the rotation set are realized by ergodic measures, there exist nonempty $f$-invariant sets $S^+$ and $S^-$ in $\T^2$ with the following property (see \cite{m-z})
$$\lim_{n\to \infty} (F^n(x)-x)/n = (\rho^\pm, \alpha) \text{ for all $x\in \pi^{-1}(S^\pm)$}.$$
Recalling that $\theta(A')$ is the projection of $A'$ into $\T^2$, we know that that $h^{-1}(J_0)\subset \theta(A')$ for some nontrivial interval $J_0$ (the projection of $J\subset \R$ into $\SS^1$). Since $R$ is an irrational rotation, $\bigcup_{n\in \Z} R^n(J_0) = \SS^1$, thus $\bigcup_{n\in\Z} f^n(\theta(h^{-1}(J_0))) = \T^2$ (from the fact that $hf=Rh$). Hence $h^{-1}(J_0)$ intersects the invariant set $S^+$, and therefore $A'$ contains some point $x$ which projects into $S^+$, which implies that $\disp_{F^n}(x)/n \to \rho^+$. The point $y\in A'$ is obtained similarly.
Since $x,y$ were chosen in $A'$ and $\gamma\prec A\prec A'$ we deduce that $A$ separates $\{x,y\}$ from $\gamma$. In addition, $\lwidth(\li f^n(A)) \geq \lwidth(\li f^n(H^{-1}(I)))$ which is uniformly bounded below by Lemma \ref{lem:lw-min} (which is stated on $\T^2$ but clearly implies this), so (5) holds (and (1) as well).

Lemma \ref{lem:density} implies that the set $G = \{ n\in \Z : \uwidth(\li f^n(A_0)) < \epsilon\}$ has density at least $1-\delta$. Thus, since $\li f^n(\gamma)$ is an essential loop in $\li f^n(A_0)$, for any $n\in G$ one has $\li f^n(A_0)\subset B_\epsilon(\li f^n(\gamma))$, and in particular (2) and (3) hold, since $\{x,y\}\subset A'\subset A_0$.
Finally, part (4) follows from Lemma \ref{lem:altura} applied to $A_0$.
\qed

\section{Topological lemmas in the annulus}\label{sec:dragging}

In this section we develop some results concerning essential loops in the annulus which under iteration remain close enough to two points having different rotation vectors. This allows to find in the sequence of iterations of the loop a sequence of arcs with increasingly large \emph{winding} number. This will be a the key point for proving Theorem I.

The \emph{winding} of an arc $\sigma\colon [a,b]\to \A$ is the number $\wind(\sigma) = \pr_1(\til \sigma(b) - \til \sigma(a))$ where $\til \sigma\colon [a,b]\to \R^2$ is a lift of $\sigma$ and $\pr_1$ denotes the projection onto the first coordinate. This number is independent of the choice of the lift. The \emph{homotopical diameter} $\diamup(\sigma)$ is the diameter of the projection of $\til \sigma$ onto the first coordinate, which again is independent of the lift.
The following simple remarks will be used:
\begin{itemize}
\item If $\sigma_1, \sigma_2$ are two arcs which can be concatenated, then $$\wind(\sigma_1*\sigma_2) = \wind(\sigma_1)+\wind(\sigma_2);$$
\item If $\sigma_1$ and $\sigma_2$ are homotopic with fixed endpoints, then $\wind(\sigma_1) = \wind(\sigma_2)$;
\item $\diamup(\sigma) = \sup_{\sigma'} \abs{\wind(\sigma')}$, where the supremum runs over all subarcs $\sigma'$ of $\sigma$.
\item If $\gamma$ is a simple loop, then $\abs{\wind(\gamma)}\leq 1$.
\end{itemize}

Recall the definition of the horizontal displacement function $\disp_F\colon \A \to \R$ from Section \ref{sec:prelim}.
\begin{lemma}\label{lem:wind-disp} If $\li f\in \homeo_0(\A)$ is a homeomorphism isotopic to the identity with a lift $F\colon \R^2\to \R^2$, for any arc $\sigma$ in $\A$ joining $x$ to $y$, $$\wind(\li f(\sigma)) = \wind(\sigma) + \disp_F(y)-\disp_F(s).$$
\end{lemma}
\begin{proof} It suffices to note that if $\til{\sigma}$ is a lift of $\sigma$ to $\R^2$ joining $\til x$ to $\til y$, then $F(\til \sigma)$ is a lift of $\li f(\sigma)$ and its endpoints are $F(\til x)$ and $F(\til y)$, so $\wind(\li f(\sigma)) = \pr_1(F(\til y)) - \pr_1(F(\til x)) = \pr_1(\til y) - \pr_1(\til x) + \disp_F(y) - \disp_F(x)$ and the claim follows.
\end{proof}

\begin{lemma}\label{lem:wind-diam} Suppose $\sigma$ is a simple arc, and $\eta$ is any arc disjoint from $\alpha$ except at their two endpoints, which coincide. Then $\abs{\wind(\sigma)} \leq \diamup(\eta) + 1$.
\end{lemma}
\begin{proof}
We may assume that $\eta$ is a simple arc by choosing a simple arc in its image joining the same two endpoints. Since $\eta$ and $\sigma$ are simple arcs intersecting only at their endpoints, after a change in orientation of $\eta$ if necessary we have that $\eta^{-1} * \sigma$ is a simple loop. This means that $\abs{\wind(\sigma) - \wind(\eta)} = \abs{\wind(\eta^{-1}*\sigma)} \leq 1$. Hence $\abs{\wind(\sigma)} \leq 1 + \abs{\wind(\eta)} \leq 1 + \diamup(\eta)$ as claimed.
\end{proof}

\begin{lemma}\label{lem:wind-2} Suppose $\alpha, \beta$ are two disjoint simple arcs. Let $\sigma_1$ be an arc joining the initial point of $\alpha$ to the initial point of $\beta$ and otherwise disjoint from $\alpha$ and $\beta$, and $\sigma_2$ an arc joining the final point of $\alpha$ to the final point of $\beta$ and otherwise disjoint from $\alpha$ and $\beta$. Then $$\abs{\wind(\alpha) - \wind(\beta)} \leq 2\diamup(\sigma_1)+2\diamup(\sigma_2) +2.$$
\end{lemma}
\begin{proof}
We may assume that $\sigma_1$ and $\sigma_2$ are simple arcs by choosing a simple arc in their images joining the same two endpoints. Suppose first that $\sigma_1$ intersects $\sigma_2$. In that case, we may choose an arc $\sigma$ in the union of their images, joining the final point of $\alpha$ to its initial point. Since $\sigma$ is disjoint from $\alpha$ except at its two endpoints, the previous lemma implies $$\abs{\wind(\alpha)} \leq \diamup(\sigma)+1 \leq \diamup(\sigma_1)+\diamup(\sigma_2)+1.$$ A similar argument shows that $\abs{\wind(\beta)} \leq \diamup(\sigma_1)+\diamup(\sigma_2)+1$, and the claim follows.

Now assume that $\sigma_1$ and $\sigma_2$ are disjoint. Then since they are also disjoint from $\alpha$ and $\beta$ except at their endpoints, it follows that $\alpha*\sigma_2*\beta*\sigma_1^{-1}$ is a simple loop, hence
$\abs{\wind(\alpha*\sigma_2*\beta^{-1}*\sigma_1^{-1})}\leq 1$.
This implies that $$\abs{\wind(\alpha) + \wind(\sigma_2) - \wind(\beta) - \wind(\sigma_1)} \leq 1,$$
and so
$$\abs{\wind(\alpha) - \wind(\beta)} \leq 1 + \abs{\wind(\sigma_1)} + \abs{\wind(\sigma_2)} \leq 1 + \diamup(\sigma_2)+\diamup(\sigma_1)$$
which implies the claim of the lemma.
\end{proof}

The following is a key lemma. Although we give a general statement, we will be interested in the case where an essential loop remains close to two points having different rotation vectors.

\begin{lemma}[Dragging lemma]\label{lem:dragging} Suppose that $\li f\colon \A\to \A$ is isotopic to the identity, and let $\gamma\subset \A$ be a simple loop. Given $x,y\in \A$, let $\sigma_x, \sigma_y$ be two simple arcs joining $x,y$ to $\gamma$ and disjoint from $\gamma$ except at their endpoints $x_0, y_0$, respectively. Similarly let $\sigma_{f(x)}, \sigma_{f(y)}$ be two simple arcs joining $\li f(x), \li f(y)$ to $\li f(\gamma)$ and disjoint from $\gamma$ except at their endpoints $x_1, y_1$, respectively. Let $I$ be a simple arc in $\gamma$ joining $x_0$ to $y_0$ and $I'$ a simple arc in $\li f(\gamma)$ joining $x_1$ to $y_1$. Then
$$\abs{\wind(I') - \wind(\li f(I))} \leq 3 + \diamup(\sigma_{f(x)}) + \diamup(\li f(\sigma_x)) + \diamup(\sigma_{f(y)}) + \diamup(\li f(\sigma_y)).$$
\end{lemma}
\begin{proof}
Fix a point $z$ in $\li f(\gamma)$ disjoint from $\li f(I)$. Let $\alpha_x$ be a simple arc in $\li f(\gamma)$ joining $x_1$ to $\li f(x_0)$ not containing $z$, and $\alpha_y$ a simple arc in $\li f(\gamma)$ joining $y_1$ to $\li f(y_0)$ not containing $z$. Note that the endpoints of $\alpha_x$ are connected by the arc $\sigma_{f(x)}^{-1}*\li f(\sigma_x)$, which is disjoint from $\alpha_x$ except at its endpoints. Thus from Lemma \ref{lem:wind-diam}, we have $\wind(\alpha_x)\leq \diamup(\sigma_{f(x)}^{-1}*\li f(\sigma_x)) \leq \diamup(\sigma_{f(x)}) + \diamup(\li f(\sigma_x))+1$.
A similar argument shows that $\wind(\alpha_y)\leq \diamup(\sigma_{f(y)})+\diamup(\li f(\sigma_x)) + 1$.
Note that  $\alpha_y^{-1}*\li f(I)*\alpha_x$ is an arc contained in the simple loop $\li f(\gamma)$ and does not contain the point $z$, so it is homotopic (with fixed endpoints) to a simple arc $J$ in $\li f(\gamma)$ joining $x_1$ to $y_1$, and we have $\wind(J) = \wind(\alpha_y) + \wind(\li f(I)) + \wind(\alpha_x)$. Note that the arc $I'$ from the statement and $J$ are both simple subarcs of the simple loop $\li f(\gamma)$ joining the same points, so there are two possibilities: $I' = J$ or $I'$ is the complementary arc of $J$ in $\gamma$. In the first case, we have $\wind(I') = \wind(J)$, and in the latter case $I'*J^{-1}$ is a simple loop, so $\abs{\wind(I') - \wind(J)} = \abs{\wind(I'*J^{-1})}\leq 1$. In both cases, we have
$$\abs{\wind(I') - \wind(\li f(I))} \leq \abs{\wind(I') - \wind(J)} + \abs{\wind(J) - \wind(\li f(I))} \leq 1 + \abs{\wind(\alpha_x)} + \abs{\wind(\alpha_y)}$$
and the desired inquality follows.
\end{proof}

\subsection{Distortion of loops and annuli}

If $\gamma\subset \A$ is an essential loop, we define its distortion as
$$\dist(\gamma) = \sup \{\diamup(\sigma): \sigma \text{ is a simple arc in $A$} \}$$
If $A\subset \A$ is an essential closed topological annulus we define its distortion as
$$\dist(A) = \inf \{\dist(\gamma): \gamma \text{ is an essential loop in $A$} \}.$$

The lower width of an annulus in a compact region of $\A$ is related to its distortion by the next lemma.
\begin{lemma}\label{lem:dist-lw} If $A\subset \A$ is an essential closed topological annulus and $\dist(A)>1$, then
$$\lwidth(A) \leq \frac{\diam(\pr_2(A))}{\dist(A)-1}.$$
\end{lemma}
\begin{remark} With some additioanl work, one may improve the bound on the right hand side to $\diam(\pr_2(A))/(2\dist(A)-1)$, but we leave the details to the reader as we will not need this fact.
\end{remark}
\begin{proof}
Let $M= \dist(A)$, and fix a vertical line $L$ in $\A$. Let $\mc{I}$ be the family of all connected components of $A\cap L$ which connect two points from different boundary components of $A$. We claim
that the number of elements of $\mc{I}$ is bounded below by $dist(A)-1$. To show this, let $m$ be the number of elements of $\mc I$ (which we assume finite, otherwise there is nothing to be done).
Choose an essential loop $\gamma$ in $A$ intersecting each element of $\mc I$ exactly once.
If $\alpha$ is any simple subarc of $\gamma$ with $\diamup(\alpha)\geq 1$, then $\alpha$ is a concatenation of arcs $\alpha_0*\alpha_1*\cdots*\alpha_k$ such that each $\alpha_i$ is disjoint from $L$ except perhaps at its endpoints, and when $0< i \leq k$ the initial point of $\alpha_i$ belongs to some element of $\mc I$. Since $\alpha_i$ is simple and contained in $\gamma$, each element of $\mc I$ appears as the initial point of at most one $\alpha_i$. This implies that $k\leq m$.
From the fact that $\alpha_i$ is disjoint from $L$ except at most at its endpoints, we deduce that $\diamup(\alpha_i)\leq 1$. Since $\alpha$ is the concatenation of the arcs $\alpha_i$, we have $\diamup(\alpha)\leq k+1 \leq m+1$, and taking the supremum among all such arcs $\alpha$ we obtain $\dist(\gamma) \leq m+1$. Thus $m\geq \dist(\gamma)-1\geq \dist(A)-1$ as claimed.

Finally, since the elements of $\mc{I}$ are pairwise disjoint intervals in the vertical line $L$, their total length is at most $\diam(\pr_2(A))$, and since there are at least $\dist(A)-1$ such elements, there must exist some $I\in \mc{I}$ of length at most $\ell = \diam(\pr_2(A))/(\dist(A)-1)$. Since the endpoints of $I$ are in different connected components of $\bd A$, it follows from the definition of lower width that $\lwidth(A)\leq \ell$, as claimed.
\end{proof}

\section{Proof of the main theorem}
Let $f\in \homeo_0(\T^2)$ be a horizontal extension of an irrational rotation, let $\li f\colon \A\to \A$ be a lift of $f$ and let $F\colon \R^2\to \R^2$ be a lift of $\li f$ (which also lifts $f$). Suppose for a contradiction that $\rho(F)$ is not a singleton, so it is an interval of the form $[\rho^-, \rho^+]\times \{\alpha\}$ where $\rho^+ > \rho^-$. Since $\rho(F^n) = n\rho(F)$, replacing $f$ by some power of $f$ if necessary we may assume that $\rho^+ - \rho^- \geq 10$.

Fix $0<\epsilon<1/4$ such that whenever $z_1,z_2\in \R^2$ satisfy $d(z_1,z_2)<\epsilon$ one has $d(F(z_1), F(z_2)) < 1/4$. Note that this implies that for $z_1, z_2\in \A$,
\begin{equation}\label{eq:eps}
\text{if } d(z_1,z_2)<\epsilon \text{ then } \abs{\disp_F(z_1) - \disp_F(z_2)} < 1/4 + \epsilon.
\end{equation}
Moreover, we remark that for any arc $\sigma$ in $\A$,
\begin{equation}\label{eq:eps2}
\text{if $\diamup(\sigma)<\epsilon$ then $\diamup(\li f(\sigma)) < 1/4$}
\end{equation}

Fix $\delta>0$, and let $A\subset \A$, $x,y\in \A\sm A$, $G\subset \N$, and the essential loop $\gamma\subset \A\sm A$ be as in Lemma \ref{lem:todo}.
Let $K = \max_{z\in \A} \abs{\disp_F(z)}$, which is finite since $\disp_F$ is continuous and $T_2$-periodic. Note that for any arc $\sigma$ in $\A$ one has
\begin{equation}\label{eq:K}
\diamup(\li f(\sigma)) \leq \diamup(\sigma) + 2K
\end{equation}

We will show that $\dist(\li f^n(A)) \to \infty$ as $n\to \infty$.
For each $n\in G$ fix a geodesic arc $\sigma_{x,n}$ such that $\sigma_{x,n}$ joins $\li f^n(x)$ to a point $x_n$ of $\li f^n(\gamma)$ minimizing the distance from $x$ to $\li f^n(\gamma)$, and similarly let $\sigma_{y,n}$ be a geodesic arc joining $\li f^n(y)$ to a point $y_n$ of $\li f^n(\gamma)$ minimizing the distance from $y$ to $\li f^n(\gamma)$. Note that both arcs are disjoint from $\li f^n(\gamma)$ except for their endpoints $x_n, y_n$, and $\diamup(\sigma_{x,n}) \leq d(\li f^n(x), \li f^n(\gamma))$ (and similarly for $\sigma_{y,n}$).

For each $n\geq 0$, let $I_n$ be a simple arc in $\gamma$ joining $x_n$ to $y_n$. Note that if $n\in G$, from Lemma \ref{lem:todo} we have $\diamup(\sigma_{x,n})<\epsilon<1/4$, so by (\ref{eq:eps2}) we have $\diamup(\li f(\sigma_{x,n}))<1/4$. An analogous estimate holds for $\sigma_{y,n}$.
Thus from Lemma \ref{lem:dragging} we have
$$\wind(I_{n+1}) \geq \wind(\li f(I_n)) - 7$$
and so from Lemma \ref{lem:wind-disp},
$$\wind(I_{n+1}) \geq \wind(I_n) + \disp_F(y_n) - \disp_F(x_n) - 7.$$
Noting that $d(x_n, \li f^n(x)) < \epsilon$ when $n\in G$, from (\ref{eq:eps}) we have $$\abs{\disp_F(x_n) - \disp_F(\li f^n(x))} < 1/4 + \epsilon < 1/2,$$ and similarly for $y$. Thus,
\begin{equation}\label{eq:inG}
\wind(I_{n+1}) \geq \wind(I_n) + \disp_F(\li f^n(y)) - \disp_F(\li f^n(x)) - 8.
\end{equation}
On the other hand, if $n\notin G$ we may obtain a rougher estimate: Lemma \ref{lem:todo}(4) implies $\diamup(\sigma_{x,n})\leq B$, so by (\ref{eq:K}) we have $\diamup(\li f(\sigma_{x,n})) \leq 2K + B$, hence again from Lemma \ref{lem:dragging}
$$\wind(I_{n+1}) \geq \wind(\li f(I_n)) - (3 + 4B + 4K),$$
and from Lemma \ref{lem:wind-disp},
$$\wind(I_{n+1}) \geq  \wind(I_n) + \disp_F(y_n) - \disp_F(x_n) - (3 + 4B + 4K).$$
Since $|\disp_F(x_n) - \disp_F(\li f^n(x))|\leq 2K$, we conclude
\begin{equation}\label{eq:notinG}
\wind(I_{n+1}) \geq  \wind(I_n) + \disp_F(\li f^n(y)) - \disp_F(\li f^n(x)) - (3 + 4B + 6K).
\end{equation}
Combining (\ref{eq:inG}) and (\ref{eq:notinG}) we obtain
$$\wind(I_n) \geq  \wind(I_0)  -8r_n - (3 + 4B + 6K)(n-r_n) + \sum_{k=0}^{n-1} \disp_F(\li f^k(y)) - \disp_F(\li f^k(x))$$
where $r_n$ is the cardinality of $G\cap \{1,2,\dots,n\}$. Note that the summation above is the same as $\disp_{F^n}(y) - \disp_{F^n}(x)$. Thus, using the fact that $\lim_{n\to \infty} \disp_{F^n}(y)/n - \disp_{F^n}(x)/n = \rho^+-\rho^- \geq 10$ and that the density of $G$ is at last $1-\delta$, we have
$$\liminf_{n\to \infty} \wind(I_n)/n \geq -8(1-\delta) - (3-4B + 6K)\delta + 10.$$
Recalling that the constants $K$ and $B$ depend only on $f$ and not on $\delta$, we may fix $\delta<(3-4B + 6K)^{-1}$ to conclude that
$$\liminf_{n\to \infty} \wind(I_n)/n \geq 1.$$
In particular, $\wind(I_n)\to \infty$ as $n\to \infty$.
Now let $\gamma'\subset \li f^n(A)$ be any essential simple loop. Recall from Lemma \ref{lem:todo}(5) that $A$ separates $\gamma$ from $\{x,y\}$, so $\li f^n(A)$ separates $\li f^n(\gamma)$ from $\{\li f^n(x),\li f^n(y)\}$.  This implies that any arc joining $\li f^n(x)$ or $\li f^n(y)$ to a point of $\li f^n(\gamma)$ intersects every essential loop in $A$. In particular, the arcs $\sigma_{x,n}$ and $\sigma_{y,n}$ must intersect $\gamma'$. Let $\sigma_{x,n}'$ and $\sigma_{y,n}'$ be simple subarcs of $\sigma_{x,n}$ and $\sigma_{y,n}$ joining $x_n$ to a point $x'$ of $\gamma'$ and $y_n$ to a point $y'$ of $\gamma'$, and otherwise disjoint from $\gamma'$. Denoting by $I'$ a simple subarc of $\gamma'$ joining $x'$ to $y'$, we have from Lemma \ref{lem:wind-2} that
$$\abs{\wind(I_n) - \wind(I')} \leq 2\diamup(\sigma_{x,n}') + 2\diamup(\sigma_{y,n}') + 2.$$
Since $\diamup(\sigma_{x,n}')\leq B$ and $\diamup(\sigma_{y,n}')\leq B$, we conclude that
$$\wind(I') \geq \wind(I_n) - 4B - 2.$$
Thus
$$\diamup(\gamma') \geq \abs{\wind(I')} \geq \abs{\wind(I_n) - 4B - 2}.$$
Since this estimate is independent of the choice of the loop $\gamma'$ in $\li f^n(A)$, we conclude that
$$\dist(\li f^n(A)) \geq \abs{\wind(I_n) - 4B - 2} \to \infty$$
as $n\to\infty$.
But then, recalling that $\diam(\pr_2(\li f^n(A))) \leq B$, Lemma \ref{lem:dist-lw} implies that $\lwidth(\li f^n(A))\to 0$ as $n\to \infty$, contradicting Lemma \ref{lem:todo}(5). This completes the proof.
\qed

\bibliographystyle{koro}
\bibliography{bib}

\end{document}